\documentclass[leqno,10pt]{article} 

\usepackage[colorlinks=true,linkcolor=blue,citecolor=blue]{hyperref}
\usepackage{multicol}

\def\line#1{\hbox to \hsize{#1\hfill}}

\usepackage{amsmath, amssymb}
\usepackage{amsthm}

\theoremstyle{plain} 
\newtheorem{theorem}{\indent\sc Theorem}[section]

\newtheorem{corollary}[theorem]{\indent\sc Corollary}
\newtheorem{proposition}[theorem]{\indent\sc Proposition}

\theoremstyle{definition}

\newtheorem{remark}[theorem]{\indent\sc Remark}



\title{On Hamiltonian minimal submanifolds in the space of oriented geodesics in real space forms}

\author{Nikos Georgiou$^{*}$, G. A. Lobos} 
\date{29 November 2014} 

\begin{document}

\maketitle



\footnote{ $^{*}$The author is partially supported by Fapesp (2010/08669-9).}

\begin{abstract}
 We prove that a deformation of a hypersurface in a $(n+1)$-dimensional real space form ${\mathbb S}^{n+1}_{p,1}$ induce a Hamiltonian variation of the normal congruence in the space ${\mathbb L}({\mathbb S}^{n+1}_{p,1})$ of oriented geodesics. As an application, we show that every Hamiltonian minimal submanifold in ${\mathbb L}({\mathbb S}^{n+1})$ (resp. ${\mathbb L}({\mathbb H}^{n+1})$) with respect to the (para-) K\"ahler Einstein structure is locally the normal congruence of a hypersurface $\Sigma$ in ${\mathbb S}^{n+1}$ (resp. ${\mathbb H}^{n+1}$) that is a critical point of the functional ${\cal W}(\Sigma)=\int_\Sigma\left(\Pi_{i=1}^n|\epsilon+k_i^2|\right)^{1/2}$, where $k_i$ denote the principal curvatures of $\Sigma$ and $\epsilon\in\{-1,1\}$. In addition, for $n=2$, we prove that every Hamiltonian minimal surface in ${\mathbb L}({\mathbb S}^{3})$ (resp. ${\mathbb L}({\mathbb H}^{3})$) with respect to the (para-) K\"ahler conformally flat structure is locally the normal congruence of a surface in ${\mathbb S}^{3}$ (resp. ${\mathbb H}^{3}$) that is a critical point of the functional ${\cal W}'(\Sigma)=\int_\Sigma\sqrt{H^2-K+1}$ (resp. ${\cal W}'(\Sigma)=\int_\Sigma\sqrt{H^2-K-1}\; $), where $H$ and $K$ denote, respectively, the mean and Gaussian curvature of $\Sigma$.
\end{abstract}

\section{Introduction}

The space ${\mathbb L}(M)$ of oriented geodesics of a pseudo-Riemannian manifold $(M,g)$ has been of great interest for the last three decades and has been studied by different authors (see for example \cite{AGK},\cite{An1},\cite{gag},\cite{klbg},\cite{klbg1},\cite{hitch},\cite{salvai0},\cite{salvai}). When $(M,g)$ is a Riemannian symmetric space of rank one, Alekseevsky, Guilfoyle and Kligenberg have described in \cite{AGK} all possible metrics defined on ${\mathbb L}(M)$  that are invariant under the isometry group of $g$. 

In the case where $(M,g)$  is a real $(n+1)$-dimensional space form ${\mathbb S}^{n+1}_{p,1}$ of signature $(p,n+1-p)$ with constant sectional curvature one, Anciaux has showed in \cite{An1} that ${\mathbb L}({\mathbb S}^{n+1}_{p,1})$ admits a K\"ahler or a para-K\"ahler structure $(G,{\mathbb J},\Omega)$, where ${\mathbb J}$ is the complex or paracomplex structure and $\Omega$ is the symplectic structure such that the metric $G$ is Einstein and is invariant under the isometry group of $g$. In the same work, for $n=2$, Anciaux has proved that ${\mathbb L}({\mathbb M})$ admits an extra K\"ahler or para-K\"ahler structure  $(G',{\mathbb J}',\Omega)$, where ${\mathbb J}'$ is the complex or paracomplex structure, such that the invariant metric $G'$ is of neutral signature, locally conformally flat  and is invariant under the isometry group of $g$. 

The submanifold theory of ${\mathbb L}(M)$ gives some interesting informations about the submanifold theory of $M$. For example, the normal congruence (or Gauss map) of a one-parameter family of parallel hypersurfaces in $M$ is a Lagrangian submanifold (the induced symplectic structure identically) of the corresponding space of geodesics (see for example \cite{An1}). In particular, the normal congruence $L(\Sigma)$ of a Weingarten surface $\Sigma$ in ${\mathbb S}^3_{p,1}$ (its principal curvatures are functionally related) is flat with respect to the metric $G'$ induced on $L(\Sigma)$ \cite{An1}.

Let $(M,J,g,\omega)$ be a (para-) K\"ahler manifold and let $\phi:\Sigma\rightarrow M$ be a Lagrangian immersion. A normal vector field $X$ is called \emph{Hamiltonian} if $X=J\nabla u$, where $J$ is the  (para-) complex structure and $\nabla u$ is the gradient of $u\in C^{\infty}(\Sigma)$ with respect to the non-degenerate induced metric $\phi^{\ast}g$. We say that a variation $(\phi_t)$ of $\phi$ is a \emph{Hamiltonian variation} if its velocity $X=\partial_t|_{t=0}\phi_t$ is a Hamiltonian vector field with the additional condition that the function $u$ is compactly supported. The Lagrangian immersion $\phi$ is said to be \emph{Hamiltonian minimal} or \emph{$H$-minimal} if it is a critical point of the volume functional with respect to Hamiltonian variations. The first variation formula of the volume functional implies that a Hamiltonian minimal submanifold is characterised by the equation $\mbox{div} JH=0$, where $H$ denotes the mean curvature vector of $\phi$ and div is the divergence operator with respect to the induced metric  \cite{Oh1}. Further study of $H$-minimal submanifolds can be found at the following articles \cite{BC, JLS,Lee, Oh2}. 

Palmer showed in \cite{Palmer1} that a smooth variation of a hypersurface in the sphere ${\mathbb S}^{n+1}$ induces  a Hamiltonian variation of the Gauss map in ${\mathbb L}({\mathbb S}^{n+1})$. An analogue result for the  3-dimensional Euclidean space ${\mathbb E}^3$ has been shown by Anciaux, Guilfoyle and Romon in \cite{AGR}. Following similar computations that Palmer used in \cite{Palmer1}, we prove that  any smooth variation of a hypersurface in the real space form ${\mathbb S}^{n+1}_{p,1}$ induces a Hamiltonian variation in the symplectic manifold $(L^{\pm}({\mathbb S}^{n+1}_{p,1}),\Omega)$. In particular we prove the following:
\begin{theorem}\label{t:theo1}
Let $\phi_t$, $t\in (-\epsilon,\epsilon)$ be a smooth one-parameter deformation of an immersion $\phi:=\phi_{t=0}$, of the $n$-dimensional oriented manifold $\Sigma$ in the real space form ${\mathbb S}^{n+1}_{p,1}$. Then, the corresponding Gauss maps $\Phi_t$ form a Hamiltonian variaton with respect to the symplectic manifold $(L^{\pm}({\mathbb S}^{n+1}_{p,1}),\Omega)$. 
\end{theorem}
Anciaux, Guilfoyle and Romon have proved in \cite{AGR} that a Hamiltonian minimal surface in ${\mathbb L}({\mathbb E}^3)$ (resp. ${\mathbb L}({\mathbb E}_1^3)$) is the Gauss map of a surface $S$ in ${\mathbb E}^3$ (resp. ${\mathbb E}^3_1$) which is a critical point of the functional ${\cal F}=\int_{S}\sqrt{H^2-K}dA$, where $H$ and $K$ denote the mean curvature and the Gauss curvature, respectively. In this article we extend this result for the case of the space ${\mathbb L}({\mathbb S}^{n+1}_{p,1})$ of oriented geodesics in a $(n+1)$-dimensional real space form. In particular, we consider the (para-) K\"ahler  Einstein structure $(G,{\mathbb J},\Omega)$ and the locally conformally flat (para-) K\"ahler structure $(G',{\mathbb J}',\Omega)$ both endowed on ${\mathbb L}({\mathbb S}^{n+1}_{p,1})$. Then, as an application of Theorem \ref{t:theo1}, we prove the following:
\begin{theorem}\label{c:cor1}
Let $\phi:\Sigma^n\rightarrow {\mathbb S}^{n+1}_{p,1}$ be a real diagonalizable hypersurface in ${\mathbb S}^{n+1}_{p,1}$ and let $\Phi$ be the Gauss map of $\phi$. Then, away of umbilic points, we have the following statements:

$(i)$ The Gauss map $\Phi$ is a Hamiltonian minimal submanifold with respect to the \emph{(\mbox{para-})}K\"ahler Einstein structure $(G,{\mathbb J})$ if and only if the immersion $\phi$ is a critical point of the functional 
\[
{\cal W}(\phi)=\int_\Sigma\sqrt{\Pi_{i=1}^n|\epsilon+k_i^2|}\; dV,
\]
where $k_1,\ldots,k_n$ are the principal curvatures of $\phi$ and $\epsilon$ denotes the length of the normal vector field of $\phi$.

$(ii)$ For $n=2$, the Gauss map $\Phi$ is a Hamiltonian minimal surface in $(L^{\pm}({\mathbb S}^3_{p,1}),G',{\mathbb J}')$ if and only if the surface $\phi$ is a critical point of the functional 
\[
{\cal W}'(\phi)=\int_{\Sigma}|k_1-k_2|\; dA,
\]
where $k_1,k_2$ denote the principal curvatures of $\phi$.
\end{theorem}

\vspace{0.2in}

\noindent {\bf Acknowledgements.} The authors would like to thank H. Anciaux and Martin A. Magid for their helpful and valuable suggestions and comments.

\vspace{0.2in}

\section{Preliminaries}

For $n\geq 1$, consider the Euclidean space ${\mathbb R}^{n+2}$ endowed with the canonical pseudo-Riemannian metric of signature $(p,n+2-p)$, where $0\leq p\leq n+2$:
\[
\left<\cdot,\cdot\right>_p=-\sum_{i=1}^pdx_i^2+\sum_{i=p+1}^{n+2}dx_i^2.
\]
Define the $(n+1)$-dimensional real space form
\[
{\mathbb S}^{n+1}_{p,1}=\{x\in {\mathbb R}^{n+2}|\left<x,x\right>_p=1 \},
\]
and let $\iota:{\mathbb S}^{n+1}_{p,1}\hookrightarrow ({\mathbb R}^{n+2},\left<\cdot,\cdot\right>_p)$ be the canonical inclusion. The induced metric $\iota^{\ast}\left<.,.\right>_p$ has signature $(p,n+1-p)$ and is of constant sectional curvature $K=1$. 

Following the notations of \cite{An1} we denote by $L^+({\mathbb S}^{n+1}_{p,1})$ (resp. $L^-({\mathbb S}^{n+1}_{p,1})$) the set of spacelike (resp. timelike) oriented geodesics of ${\mathbb S}^{n+1}_{p,1}$, that is, 
\[
L^{\pm}({\mathbb S}^{n+1}_{p,1})=\{x\wedge y\in \Lambda^2({\mathbb R}^{n+2})\; |\; y\in T_x{\mathbb S}^{n+1}_{p,1}, \left<y,y\right>_p=\epsilon\},
\]
where $\epsilon=1$ (resp. $\epsilon=-1$) corresponds to $L^+({\mathbb S}^{n+1}_{p,1})$ (resp. $L^-({\mathbb S}^{n+1}_{p,1})$). If $\Lambda^2({\mathbb R}^{n+2})$ is equipped with the flat pseudo-Riemannian metric:
\[
\left<\left<x_1\wedge y_1,x_2\wedge y_2\right>\right>=\left<x_1,x_2\right>_p\left< y_1,y_2\right>_p-\left<x_1,y_2\right>_p\left<x_2, y_1\right>_p,
\]
we denote by $G$ the metric $\left<\left<\cdot,\cdot\right>\right>$ induced by the inclusion map $i:L^{\pm}({\mathbb S}^{n+1}_{p,1})\hookrightarrow \Lambda^2({\mathbb R}^{n+2})$. On the other hand, in $L^{\pm}({\mathbb S}^{n+1}_{p,1})$ can be defined a complex (paracomplex) structure ${\mathbb J}$ as follows:

Let $J$ be the canonical complex (paracomplex) structure  in the oriented plane $x\wedge y\in L^{\pm}({\mathbb S}^{n+1}_{p,1})$ defined by $Jx=y$ and $Jy=-\epsilon x$. Thus, $J^2=-\epsilon Id$. A tangent vector to $i(L^{\pm}({\mathbb S}^{n+1}_{p,1}))$ at the point $x\wedge y$ is of the form $x\wedge X+y\wedge Y$, where $X,Y\in (x\wedge y)^{\bot}$ in $\Lambda^2({\mathbb R}^{n+2})$. The complex (paracomplex) structure ${\mathbb J}$ is defined by:
\[
{\mathbb J}(x\wedge X+y\wedge Y):=(Jx)\wedge X+(Jy)\wedge Y=y\wedge X+\epsilon x\wedge Y.
\]
The metric $G$ and the (para) complex structure ${\mathbb J}$ are invariant under the natural action of the isometry group of ${\mathbb S}^{n+1}_{p,1}$. For $n\geq 3$, it has been showed  in \cite{AGK} that $G$ is the unique invariant metric under the natural action of $SO(n+2-p,p)$.

The 2-form $\Omega$ defined by $\Omega(\cdot,\cdot)=\epsilon G({\mathbb J}\cdot,\cdot)$, is a symplectic structure on $L^{\pm}({\mathbb S}^{n+1}_{p,1})$ and in particular:
\begin{proposition}\emph{\cite{An1}}
The quadraple $(L^+({\mathbb S}^{n+1}_{p,1}), G,{\mathbb J},\Omega)$ is a $2n$-dimensional K\"ahler manifold with signature $(2p,2n-2p)$, while $(L^-({\mathbb S}^{n+1}_{p,1}), G,{\mathbb J},\Omega)$ is a $2n$-dimensional  para-K\"ahler manifold. In both cases, the metric $G$ is Einstein with constant scalar curvature $S=2\epsilon n^2$.
\end{proposition}

We now consider the case of $L^{\pm}({\mathbb S}^3_{p,1})\subset \Lambda^2({\mathbb R}^4)$. The orthogonal $(x\wedge y)^{\bot}$ of an oriented plane $x\wedge y\in L^{\pm}({\mathbb S}^3_{p,1})$,  is also a plane in ${\mathbb R}^4$ and is oriented in such a way its orientation is combatible with the orientation of the plane $x\wedge y$. Then it is possible to define a canonical complex or paracomplex structure $J'$, depending of whether the metric $\left<\left<.,.\right>\right>$ induced on $(x\wedge y)^{\bot}$ is positive or indefinite. In this case, we may define a complex or paracomplex structure ${\mathbb J}'$ on $L^{\pm}({\mathbb S}^3_{p,1})$ by
\[
{\mathbb J}'(x\wedge X+y\wedge Y):=x\wedge (J'X)+y\wedge (J'Y).
\]
The pseudo-Riemannian metric $G'$ on $L^{\pm}({\mathbb S}^3_{p,1})$ is given by:
\[
G'(\cdot,\cdot):=\Omega(\cdot,J'\cdot)=-\epsilon G(\cdot,{\mathbb J}\circ {\mathbb J}'\cdot).
\]
Furthermore,
\begin{proposition}\emph{\cite{An1}}
The quadraples $(L^{\pm}({\mathbb S}^3_{p,1}), G',{\mathbb J}',\Omega)$ are $4$-dimensional \emph{(}para-\emph{)} K\"ahler manifolds. The metric $G'$ is of neutral signature $(2,2)$, scalar flat, locally conformally flat and is  invariant under the natural action of $SO(4-p,p)$.
\end{proposition}
Let $\phi:\Sigma^n\rightarrow {\mathbb S}^{n+1}_{p,1}$ be an immersion of a $n$-dimensional orientable manifold into the real space form ${\mathbb S}^{n+1}_{p,1}$ and let $N$ be the unit normal vector of the hypersurface $S=\phi(\Sigma)$. The set $\bar S$ of geodesics that are orthogonal to $S$, oriented in the direction of $N$, is called \emph{the normal congruence} or \emph{the Gauss map} of $S$. 
Then,
\begin{proposition}\emph{\cite{An1}}
Let $\phi$ be an immersion of an orientable manifold $\Sigma^n$ in ${\mathbb S}^{n+1}_{p,1}$ with unit normal vector field $N$. Then the Gauss map of $S=\phi(\Sigma)$ is the image of the map $\Phi:\Sigma^n\rightarrow L^{\pm}({\mathbb S}^{n+1}_{p,1})$ defined by $\Phi=\phi\wedge N$. When $\Phi$ is an immersion, it is Lagrangian. Conversely, let $\Phi:\Sigma^n\rightarrow L^{\pm}({\mathbb S}^{n+1}_{p,1})$ be an immersion of a simply connected $n$-manifold. Then $\bar S:=\Phi(\Sigma)$ is the Gauss map of an immersed hypersurface of ${\mathbb S}^{n+1}_{p,1}$ if and only if $\Phi$ is Lagrangian.
\end{proposition}

\vspace{0.2in}

\section{Hamiltonian minimal submanifolds}

Throughout the article, when we talk about hypersurfaces in ${\mathbb S}^{n+1}_{p,1}$ or about Lagrangian submanifolds in $L^{\pm}({\mathbb S}^{n+1}_{p,1})$, we mean that the induced metric is non-degenerate.

\subsection{Hamiltonian variations in $L^{\pm}({\mathbb S}^{n+1}_{p,1})$}

Consider the $2n$-dimensional symplectic manifold $(L^{\pm}({\mathbb S}^{n+1}_{p,1}),\Omega)$, where $L^{\pm}({\mathbb S}^{n+1}_{p,1})$ denotes the space of oriented geodesics in the real space form ${\mathbb S}^{n+1}_{p,1}$. Then we prove our first main result:

{\indent\sc Proof of Theorem \ref{t:theo1}:} Let $\phi_t:\Sigma^n\rightarrow {\mathbb S}^{n+1}_{p,1}$, where $t\in (-t_0,t_0)$ for some $t_0>0$, be a smooth variation of an immersion $\phi:=\phi_0:\Sigma^n\rightarrow {\mathbb S}^{n+1}_{p,1}$ of an oriented $n$-dimensional manifold $\Sigma$ in ${\mathbb S}^{n+1}_{p,1}$. Let $S_t:=\phi_t(\Sigma)$ be the hypersurfaces of ${\mathbb S}^{n+1}_{p,1}$ and $S:=\phi(\Sigma)$.

We denote by $N_t$ the 1-parameter family of vector fields such that $N_0=N$ and $\left<N_t,\phi_t\right>_p=0$. Following the notation of \cite{Palmer1}, there exist a smooth function $f$ on $S$ and a smooth section $Y$ of the tangent bundle $TS$  such that
\[
\dot\phi=f N+Y,
\]
where $\dot\phi=\partial_t\phi_t|_{t=0}$. Derivating the expression $\left<N_t,\phi_t\right>_p=0$ with respect to $t$ at $t=0$, we have 
\[
\big<\dot N,\phi\big>_p=-\big<N,\dot\phi\big>_p.
\]
Then, it yields
\[
\big<\dot N,\phi\big>_p=-\epsilon f,
\]
where $\epsilon=\big<N,N\big>_p$. For $X_t\in TS_t$, where $X_0:=X\in TS$, we also derive with respect to $t$ the expression $\left<N_t,X_t\right>_p=0$ at $t=0$, and we obtain 
\[
\big<\dot N,X\big>_p=-\epsilon df(X)+\left<dN(Y),X\right>_p,
\]
which finally gives
\[
\dot N=-\epsilon\nabla f+dN(Y)-\epsilon f\phi.
\]
 Let $\Phi:\Sigma^n\rightarrow L^{\pm}({\mathbb S}^{n+1}_{p,1}):x\mapsto \phi(x)\wedge N(x)$ be the Gauss map of the immersion $\phi:\Sigma^n\rightarrow {\mathbb S}^{n+1}_{p,1}$. If $\bar X:=d\Phi(X)$, from  \cite{An1}, we have that
\[
\bar X=X\wedge N+AX\wedge \phi.
\]
Let $\dot{\Phi}=\partial_t\Phi_t|_{t=0}$ be the velocity of the variation $\Phi_t$ and write $\dot{\Phi}=\dot{\Phi}^{\top}+\dot{\Phi}^{\bot}$, where $\dot{\Phi}^{\top}$ and $\dot{\Phi}^{\bot}$ denote the tangential and the normal component of $\dot{\Phi}$ respectively. Then,
\begin{eqnarray}
G(\dot{\Phi},{\mathbb J}\bar X)&=&G(Y\wedge N+\epsilon\nabla f\wedge\phi-dN(Y)\wedge\phi,\; {\mathbb J}(X\wedge N+AX\wedge\phi))\nonumber \\
&=&-df(X),\nonumber
\end{eqnarray}
which implies that 
\[
\dot{\Phi}^{\bot}=-{\mathbb J}\overline\nabla f,
\]
and this completes the Theorem.
$\Box$

\subsection{Applications}

A \emph{real diagonalizable immersion} is a smooth immersion $\phi$ of a hypersurface $\Sigma^n$ into the (n+1)-dimensional real space form ${\mathbb S}^{n+1}_{p,1}$ such that, locally, the shape operator $A$ can be diagonalized, that is, it is existed a local orthonormal frame $(e_1,\ldots, e_n)$ and real functions $k_1,\ldots,k_n$ where, $A=\mbox{diag}(k_1,\ldots,k_n)$.
In this case, each vector field $e_i$ is called \emph{a principal direction} with corresponded \emph{principal curvature} $k_i$. 


\begin{remark}\label{r:rem}
Note that every hypersurface in the Riemannian case,  away of umbilic points, is real diagonalizable. 
\end{remark}

We are now in position to prove our second result:

{\indent\sc Proof of Theorem \ref{c:cor1}:} Consider a smooth immersion of $\phi$ of the $n$-dimensional manifold $\Sigma$ in ${\mathbb S}^{n+1}_{p,1}$ and let $\Phi:\Sigma^n\rightarrow L^{\pm}({\mathbb S}^{n+1}_{p,1})$ be the corresponded Gauss map. The fact that $\phi$ is real diagonalizable implies the existence of an orthonormal frame $(e_1,\ldots,e_n)$, with respect to the induced metric $\phi^{\ast}g$, such that
\[
Ae_i=k_i e_i, \qquad i=1,\ldots,n,
\]
where $A$ denotes the shape operator of $\phi$. Let $(\phi_t)_{t\in (-t_0,t_0)}$ be a smooth variation of $\phi$ and $(\Phi_t)$ be the corresponded variation of the Gauss map $\Phi$. Real diagonalizability implies that the minimal polynomial of $A$ is the product of distinct linear factors. Using the fact that the variation $(\Phi_t)$ is at least $C^1$-smooth, it is possible to obtain a positive real number $t_1<t_0$ such that $\phi_t$ is real diagonalizable for every $t\in (-t_1,t_1)$. We may extend all extrinsic geometric quantities such as the shape operator $A$, the principal directions $e_i$ and the principal curvatures $k_i$ to the 1-parameter family of immersions $(\phi_t)$.
 
\vspace{0.1in}

\noindent (i) From \cite{An1}, the induced metric $\Phi_t^{\ast}G$ is given by $\Phi_t^{\ast}G=\epsilon \phi_t^{\ast}g+\phi_t^{\ast}g(A.,A.)$ and thus,
\[
\Phi^{\ast}_tG=\mbox{diag}\Big(\epsilon_1(\epsilon+k_1^2),\ldots,\epsilon_n(\epsilon+k_n^2)\Big),
\]
where $\epsilon_i=g(e_i,e_i)$. For every sufficiently small $t>0$, the volume of every Gauss map $\Phi_t$, with respect to the metric $G$, is
\begin{equation}\label{e:impeqt}
\mathop{\rm Vol}(\Phi_t)=\int_{\Sigma}\sqrt{|\mathop{\rm det}\Phi_t^{\ast}G|}dV={\cal W}(\phi_t).
\end{equation}
If $\phi$ is a critical point of the functional ${\cal W}$, we have
\[
\partial_t(\mathop{\rm Vol}(\Phi_t))=0,
\]
for any Hamiltonian variation of $\Phi$. Therefore, $\Phi$ is a Hamiltonian minimal submanifold with respect to the K\"ahler Einstein structure $(G,J)$. The converse follows directly from (\ref{e:impeqt}).

\vspace{0.1in}

\noindent (ii) Assume that $n=2$. From \cite{An1}, in terms of the orthonormal frame $(e_1,e_2)$, the induced metric $\Phi^{\ast}G'$ is  
\[
\Phi^{\ast}G'=\begin{pmatrix} 0 & \epsilon_2(k_2-k_1) \\ \epsilon_2(k_2-k_1)  & 0\end{pmatrix}.
\]
where $\epsilon_2=\phi^{\ast}g(e_2,e_2)$. Then, the volume of every Gauss map $\Phi_t$, with respect to the metric $G'$, is
\begin{equation}\label{e:impeqt1}
\mathop{\rm Vol'}(\Phi_t)=\int_{\Sigma}\sqrt{|\mathop{\rm det}\Phi_t^{\ast}G|}dV={\cal W}'(\phi_t),
\end{equation}
and thus the second statement of the Theorem follows by a similar argument with the proof of the first statement.
$\Box$

Let $\phi:\Sigma^n\rightarrow {\mathbb S}_{p,1}^{n+1}$ be a hypersurface and $N$ denotes the unit normal vector field. For $\theta\in {\mathbb R}$, consider the immersion $\phi_{\theta}:=\cos\epsilon(\theta)\phi+\sin\epsilon(\theta)N$, where $\epsilon:=|N|^2$ and $(\cos\epsilon(\theta),\sin\epsilon(\theta))=(\cos\theta,\cos\theta)$ if $\epsilon=1$ while for $\epsilon=-1$ we have $(\cos\epsilon(\theta),\sin\epsilon(\theta))=(\cosh\theta,\cosh\theta)$. The images $\phi_{\theta}(\Sigma)$ and $\phi(\Sigma)$ are called parallel hypersurfaces. It is important to mention that parallel hypersurfaces have the same Gauss map \cite{An1}.

Looking more carefully the relations (\ref{e:impeqt}) and (\ref{e:impeqt1}) we obtain the following symmetry for the functionals ${\cal W}$ and ${\cal W}'$:

\begin{corollary}
If $\phi_1$ and $\phi_2$ are parallel smooth real diagonalizable immersions of the $n$-manifold $\Sigma$ in ${\mathbb S}_{p,1}^{n+1}$ then ${\cal W}(\phi_1)={\cal W}(\phi_2)$. In the case where $n=2$, we also have that ${\cal W}'(\phi_1)={\cal W}'(\phi_2)$.
\end{corollary}

Using the Einstein (para-) K\"ahler structure $(L^{\pm}({\mathbb S}^3_{p,1}),G,J)$ we obtain the following Corollary:
\begin{corollary}\label{c:cor00}
Let $\phi:\Sigma^n\rightarrow {\mathbb S}^{n+1}_{p,1}$ be a real diagonalizable hypersurface and let $k_1,\ldots k_n$ be the principal curvatures. If $\Phi$ is the Gauss map of $\phi$, then the function $\sum_{i=1}^n\tan\epsilon^{-1}(k_i)$ is harmonic with respect to the induced metric $\Phi^{\ast}G$ if and only if $\phi$ is a critical point of the functional $\int_\Sigma\sqrt{\Pi_{i=1}^n|\epsilon+k_i^2|}$.
\end{corollary}
\begin{proof}
Let $\Phi$ be the Gauss map of $\phi$ and consider the Einstein (para-) K\"ahler structure $(G,J)$. Then , from \cite{An1}, we know that the mean curvature $\vec{H}$ of $\Phi$ is given by 
\[
\vec{H}=\frac{\epsilon}{n}J\nabla\Big(\sum_{i=1}^n\tan\epsilon^{-1}(k_i)\Big),
\]
where $\nabla$ denotes the Levi-Civita connection of the induced metric $\Phi^{\ast}G$. Then, the Corollary follows by the following relation,
\[
\mbox{div}(nJ\vec{H})=\Delta\Big(\sum_{i=1}^n\tan\epsilon^{-1}(k_i)\Big),
\]
where div and $\Delta$ denote the divergence operator and the Laplacian of $\Phi^{\ast}G$.
\end{proof}

Using the Remark \ref{r:rem}, we obtain the following two Corollaries:

\begin{corollary}\label{c:cor11}
Let $\phi:\Sigma^n\rightarrow {\mathbb S}^{n+1}\; (\mbox{resp. in the hyperbolic space}\; {\mathbb H}^{n+1})$ be a hypersurface in the sphere ${\mathbb S}^{n+1}$. Then the Gauss map $\Phi$ is a Hamiltonian minimal submanifold in the (para-) K\"ahler Einstein structure $(G,{\mathbb J})$ if and only if the hypersurface $\phi$ is a critical point of the functional ${\cal W}(\phi)=\int_\Sigma\sqrt{\Pi_{i=1}^n(1+k_i^2)}\; (\mbox{resp.}\;  {\cal W}(\phi)=\int_\Sigma\sqrt{\Pi_{i=1}^n|1-k_i^2|})$, where $k_1,\ldots,k_n$ are the principal curvatures of $\phi$.
\end{corollary}

\begin{corollary}\label{c:cor11}
Let $\phi:\Sigma\rightarrow {\mathbb S}^3\; (\mbox{resp. in the hyperbolic space}\; {\mathbb H}^3)$ be a surface in the sphere ${\mathbb S}^3$. Then the Gauss map $\Phi$ is a Hamiltonian minimal submanifold with respect to the K\"ahler, conformally flat structure $(G',{\mathbb J}')$ if and only if the surface $\phi$ is a critical point of the functional ${\cal W}'(\phi)=\int_\Sigma\sqrt{H^2-K+1}$ $(\mbox{resp.}\;  {\cal W}'(\phi)=\int_\Sigma\sqrt{H^2-K-1}\; )$, where $H,K$ denote the mean and the Gauss curvature of $\phi$.
\end{corollary}


\noindent Nikos Georgiou, UFAM, Manaus, AM, Brazil.

\noindent nikos@ime.usp.br 

\vspace{0.1in}

\noindent G. A. Lobos, UFSCar, S\~ao Carlos, SP, Brazil.

\noindent lobos@dm.ufscar.br


\begin{thebibliography}{99}
\bibitem{AGK}\textsc{D. Alekseevsky, B. Guilfoyle and W. Klingenberg}, \emph{On the geometry of spaces of oriented geodesics}, Ann. Global Anal. Geom. \textbf{40} (2011), 389--409.

\bibitem{An1}\textsc{H. Anciaux}, \em Space of geodesics of pseudo-Riemannian space forms and normal congruences of hypersurfaces, \em Trans. Amer. Math. Soc. {\bf 366} (2014), 2699--2718.

\bibitem{AGR} 
H. Anciaux, B. Guilfoyle, P. Romon, {\it Minimal submanifolds in the tangent bundle of a Riemannian surface},  J. Geometry and Physics. {\bf 61}, 237--247 (2011).

\bibitem{BC}
A. Butscher and J. Corvino, {\it Hamiltonian stationary tori in K\"ahler manifolds}, Calc. Var. Partial Differential Equations. {\bf 45}, 63--100 (2012).


\bibitem{gag}
N. Georgiou and B. Guilfoyle, {\it On the space of oriented geodesics of hyperbolic 3-space}, Rocky Mountain J. Math. {\bf 40}, 1183--1219 (2010).



\bibitem{klbg}
B. Guilfoyle and W. Klingenberg, {\it An indefinite K\"ahler metric on the space of oriented lines}, J. London Math. Soc. {\bf 72}, 497--509 (2005).

\bibitem{klbg1}
B. Guilfoyle and W. Klingenberg, {\it A neutral k\"ahler metric on the space of time-like lines in Lorentzian 3-space} (2005) math.DG/0608782.



\bibitem{hitch}
N.J. Hitchin, {\it Monopoles and geodesics}, Comm. Math. Phys. {\bf 83}, 579--602 (1982).

\bibitem{JLS}
D. Joyce, Y.-I. Lee and R. Schoen, {\it On the existence of Hamiltonian stationary Lagrangian submanifolds in symplectic manifolds}, Amer. J. Math. {\bf 133}, 1067--1092 (2011).

\bibitem{Lee}
Y.-I. Lee, {\it The existence of Hamiltonian stationary Lagrangian tori in K\"ahler manifolds of any dimension}, Calc. Var. Partial Differential Equations {\bf 45}, 231--251 (2012).

\bibitem{Oh1}
Y.-G. Oh, {\it Second variation and stabilities of minimal lagrangian submanifolds in K\"ahler manifolds}, Invent. Math. {\bf 101}, 501--519  (1990) 

\bibitem{Oh2}
Y.-G. Oh, {\it Volume minimization of Lagrangian submanifolds under Hamiltonian deformations}, Math. Z.  {\bf 212}, 175--192  (1993).

\bibitem{Palmer1}
B. Palmer, {\it Buckling eigenvalues, Gauss maps and Lagrangian susbmanifolds}, Diff. Geom. Appl., {\bf 4}, 391--403 (1994).


\bibitem{salvai0}
M. Salvai, {\it On the geometry of the space of oriented lines in Euclidean space}, Manuscripta Math. {\bf 118}, 181--189 (2005).

\bibitem{salvai}
M. Salvai, {\it On the geometry of the space of oriented lines of hyperbolic space}, Glasg. Math. J.
{\bf 49}, 357--366 (2007).






\end{thebibliography}
\end{document}